\documentclass[11pt]{amsart}

\title[The Quantum Lefschetz Principle for Vector Bundles]{The Quantum Lefschetz Principle for Vector Bundles\\ as a Map Between Givental Cones}

\author[Coates]{Tom Coates}
\address{Department of Mathematics\\
Imperial College London\\
180 Queen's Gate\\
London SW7 2AZ 
\\UK}
\email{t.coates@imperial.ac.uk}



\usepackage{amsfonts}
\usepackage{amssymb}
\usepackage{xypic}
\usepackage[margin=2.5cm]{geometry}
\usepackage{hyperref}
\usepackage{cite}
\usepackage{longtable}

\allowdisplaybreaks[3]
\newcommand{\CC}{\mathbb{C}}
\newcommand{\ZZ}{\mathbb{Z}}
\newcommand{\NN}{\mathbb{N}}
\newcommand{\QQ}{\mathbb{Q}}
\newcommand{\cO}{\mathcal{O}}
\newcommand{\bt}{\mathbf{t}}
\newcommand{\vir}{\text{\rm vir}}
\DeclareMathOperator{\Res}{Res}
\DeclareMathOperator{\ev}{ev}
\renewcommand{\cH}{\mathcal{H}}
\renewcommand{\cL}{\mathcal{L}}
\newcommand{\correlator}[1]{\left \langle #1 \right \rangle}
\newcommand{\bc}{{\boldsymbol{c}}}
\newcommand{\be}{{\boldsymbol{e}}}
\newcommand{\bJ}{{\mathbf{J}}}

\theoremstyle{plain}
\newtheorem{theorem}{Theorem}[section]
\newtheorem{proposition}[theorem]{Proposition}

\theoremstyle{definition}
\newtheorem{remark}[theorem]{Remark}
\newtheorem{keyremark}[theorem]{Key Remark}

\begin{document}

\maketitle

\begin{abstract}
  Givental has defined a Lagrangian cone in a symplectic vector space which encodes all genus-zero Gromov--Witten invariants of a smooth projective variety $X$.  Let $Y$ be the subvariety in $X$ given by the zero locus of a regular section of a convex vector bundle. We review arguments of Iritani, Kim--Kresch--Pantev, and Graber, which give a very simple relationship between the Givental cone for $Y$ and the Givental cone for Euler-twisted Gromov--Witten invariants of $X$.  When the convex vector bundle is the direct sum of nef line bundles, this gives a sharper version of the Quantum Lefschetz Hyperplane Principle.
\end{abstract}

\section{Gromov--Witten Invariants and Twisted Gromov--Witten Invariants}
\label{sec:GW}

Given a smooth projective variety $X$, one can define \emph{Gromov--Witten invariants} of $X$ \cite{Kontsevich:enumeration,Kontsevich--Manin}:
\begin{equation}
  \label{eq:GW_invariant}
  \correlator{\gamma_1 \psi_1^{k_1},\ldots,\gamma_n \psi_n^{k_n}}_{g,n,d}^X := 
  \int_{[X_{g,n,d}]^\vir} 
  \prod_{i=1}^{i=n} \ev_i^\star \gamma_i \cup \psi_i^{k_i}
\end{equation}
Notation here is by now standard; a list of notation and definitions can be found in Appendix~\ref{sec:notation}.  Given a class $A \in H^\bullet(X_{g,n,d};\QQ)$, we can include it in the integral \eqref{eq:GW_invariant}, writing:
\begin{equation}
  \label{eq:GW_invariant_with_extra_insertion}
  \correlator{\gamma_1 \psi_1^{k_1},\ldots,\gamma_n \psi_n^{k_n} ; A}_{g,n,d}^X := 
  \int_{[X_{g,n,d}]^\vir} 
  A \cup \prod_{i=1}^{i=n} \ev_i^\star \gamma_i \cup \psi_i^{k_i}
\end{equation}
In particular, we can consider \emph{twisted Gromov--Witten invariants} \cite{Coates--Givental}.  Let $E \to X$ be a vector bundle, and let $\bc(\cdot)$ be an invertible multiplicative characteristic class.  We can evaluate $\bc$ on classes in K-theory by setting $\bc(A \ominus B) = \frac{\bc(A)}{\bc(B)}$.  The twisting class $E_{g,n,d} \in K^0(X_{g,n,d})$ \label{def:twisting_class} is defined by $E_{g,n,d} = \pi_! \ev^\star E$, where 
\[
\xymatrix{
  C \ar[r]^-{\ev} \ar[d]_\pi & X \\
  X_{g,n,d}
}
\]
is the universal family over the moduli space of stable maps.  $(\bc,E)$-twisted Gromov--Witten invariants of $X$ are intersection numbers of the form:
\begin{equation}
  \label{eq:twisted_GW_invariant}
  \correlator{\gamma_1 \psi_1^{k_1},\ldots,\gamma_n \psi_n^{k_n} ; \bc(E_{g,n,d})}_{g,n,d}^X 
\end{equation}
Consider the $S^1$-action \label{def:S1_action} on vector bundles $V \to B$ which rotates the fibers of $V$ and leaves the base $B$ invariant.  The $S^1$-equivariant Euler class $\be(\cdot)$ is invertible over the field of fractions $\QQ(\lambda)$ of $H^\bullet_{S^1}\big(\{\text{point}\}\big) = \QQ[\lambda]$.  Taking $\bc = \be$, we refer to twisted Gromov--Witten invariants \eqref{eq:twisted_GW_invariant} as Euler-twisted Gromov--Witten invariants. 

Givental has defined a Lagrangian cone $\cL_X$ in a symplectic vector space $\cH_X$ which encodes all genus-zero Gromov--Witten invariants of $X$ \cite{Givental:quadratic,Givental:symplectic}.  Fix a basis $\{\phi_\epsilon\}$ for $H^\bullet(X;\QQ)$, and let $\{\phi^\epsilon\}$ denote the dual basis with respect to the Poincar\'e pairing $(\cdot,\cdot)$ on $H^\bullet(X)$, so that $(\phi_\mu,\phi^\nu) = \delta_\mu^\nu$.  Let $\Lambda_X$ denote the Novikov ring of $X$; this is defined in Appendix~\ref{sec:notation}.  Consider the vector space (or rather, free $\Lambda_X$-module):
\[
\label{def:HX}
\cH_X := H^\bullet(X;\Lambda_X) \otimes \CC(\!(z^{-1})\!)
\]
equipped with the symplectic form (or rather, $\Lambda_X$-valued symplectic form):
\[
\label{def:Omega}
\Omega_X(f,g) := \Res_{z=0} \big(f(-z),g(z)\big) \, dz
\]
Let $\bt(z) = t_0 + t_1 z + t_2 z^2 + \cdots$, where $t_i \in H^\bullet(X;\Lambda_X)$.   A general point on Givental's Lagragian cone $\cL_X \subset \cH_X$ has the \label{def:LX} form:
\begin{equation}
  \label{eq:JX}
  \bJ_X(\bt) := 
  {-z} + \bt(z) + \sum
  \frac{Q^d}{n!} 
  \correlator{t_{k_1} \psi_1^{k_1},\ldots,t_{k_n} \psi_n^{k_n}, \phi^\epsilon \psi_{n+1}^m}^X_{0,n+1,d} \phi_\epsilon \, (-z)^{-m-1}
\end{equation}
where the sum runs over non-negative integers~$n$ and~$m$, multi-indices $k = (k_1,\ldots,k_n)$ in $\NN^n$, degrees $d \in H_2(X;\ZZ)$, and basis indices $\epsilon$.  Knowing the Lagrangian submanifold $\cL_X$ is equivalent to knowing all genus-zero Gromov--Witten invariants \eqref{eq:GW_invariant} of $X$.

A similar Lagrangian cone encodes all genus-zero Euler-twisted Gromov--Witten invariants of $X$.  Consider the twisted Poincar\'e pairing $(\alpha,\beta)_\be = \int_X \alpha \cup \beta \cup \be(E)$, and the twisted symplectic form:
\[
\label{def:Omega_e}
\Omega_\be(f,g) := \Res_{z=0} \big(f(-z),g(z)\big)_\be \, dz
\]
on $\cH_X$.  Let $\{\phi^\epsilon_\be\}$ denote the basis dual to $\{\phi_\epsilon\}$ with respect to the twisted Poincar\'e pairing, so that $(\phi_\mu,\phi^\nu_\be)_\be = \delta_\mu^\nu$.  A general point on the Lagrangian cone $\cL_\be \subset \big(\cH_X,\Omega_\be\big)$ has the   \label{def:LX_e} form:
\begin{equation}
  \label{eq:Je}
  \bJ_\be(\bt) := {-z} + \bt(z) + \sum
  \frac{Q^d}{n!} 
  \correlator{t_{k_1} \psi_1^{k_1},\ldots,t_{k_n} \psi_n^{k_n}, \phi^\epsilon_\be \psi_{n+1}^m; \be(E_{0,n+1,d})}^X_{0,n+1,d} \phi_\epsilon \, (-z)^{-m-1}
\end{equation}
where the sum runs over the same set as above.  Knowing $\cL_\be$ is equivalent to knowing all genus-zero Euler-twisted Gromov--Witten invariants of $X$. In this expository note, we describe a close relationship, in the case where the vector bundle $E$ is convex, between Euler-twisted invariants of $X$ and Gromov--Witten invariants of the subvariety $Y \subset X$ defined by a regular section of $E$.  We prove:

\begin{theorem}
  \label{thm:main}
  Let $X$ be a smooth projective variety.  Let $E \to X$ be a convex vector bundle,  let $Y$ be the subvariety in $X$ defined by a regular section of $E$, and let $i \colon Y \to X$ be the inclusion map.  Let $\bJ_\be$ denote the general point \eqref{eq:Je} on the Lagrangian cone $\cL_\be$ for Euler-twisted Gromov--Witten invariants of $X$.  Let $\bJ_Y$ denote the general point on the Lagrangian cone $\cL_Y$ for genus-zero Gromov--Witten invariants of $Y$, as in \eqref{eq:JX}.  Then the non-equivariant limit $\bJ_\be\big|_{\lambda = 0}$ is well-defined and satisfies:
\[
i^\star \bJ_\be(\bt)\big|_{\lambda=0} = \bJ_Y(i^\star \bt)
\]
In particular, $i^\star \cL_\be\big|_{\lambda = 0} \subset \cL_Y$.  

\noindent Throughout here we have applied the homomorphism $Q^\delta \mapsto Q^{i_\star \delta}$ to the Novikov ring of $Y$.
\end{theorem}

\begin{remark}
  A vector bundle $E \to X$ is called convex if and only if $H^1(C,f^\star E) = 0$ for all stable maps $f \colon C \to X$ such that the curve $C$ has genus zero. Globally generated vector bundles are automatically convex, as are direct sums of nef line bundles.
\end{remark}

\begin{remark}
  If the dimension of $Y$ is at least $3$ then, by the Lefschetz theorem, the homomorphism of Novikov rings $\Lambda_Y \to \Lambda_X$ given by $Q^\delta \mapsto Q^{i_\star \delta}$ is an isomorphism.
\end{remark}

\begin{remark}
  In the non-equivariant limit, the map $i^\star \colon \cH_X \to \cH_Y$ becomes symplectic: it satisfies $i^\star \Omega_\be\big|_{\lambda=0} = \Omega_Y$.  Thus Theorem~\ref{thm:main} fits neatly into a general story that encompasses the Crepant Resolution Conjecture~\cite{Coates--Iritani--Tseng,Coates--Ruan}, Brown's toric bundle theorem~\cite{Brown}, and so on: geometrically-natural operations in Gromov--Witten theory give rise to symplectic transformations of Givental's symplectic space that preserve the Lagrangian cones.
\end{remark}

\begin{keyremark}
Only the statement of Theorem~\ref{thm:main} is new.  As we will see, the proof is a very minor variation of an argument by Iritani~\cite[Proposition~2.4]{Iritani:periods}.  Iritani's result in turn builds on arguments by Kim--Kresch--Pantev~\cite{Kim--Kresch--Pantev} and Graber~\cite[\S2]{Pandharipande:after}.
\end{keyremark}

\begin{remark}
  Theorem~\ref{thm:main} improves upon \cite[formula~19]{Coates--Givental}, which roughly speaking, in the special case where $E$ is the direct sum of nef line bundles, relates $\bJ_\be(\bt)\big|_{\lambda=0}$ to $i_\star \bJ_Y(i^\star \bt)$.  The improved version determines invariants of $Y$ with one insertion (that at the last marked point) involving an arbitrary cohomology class on $Y$, whereas the original version determined only invariants of $Y$ such that all insertions are pullbacks of cohomology classes on $X$.  When combined with the Lee--Pandharipande reconstruction theorem~\cite{Lee--Pandharipande} this determines, under moderate hypotheses on $Y$, the big quantum cohomology of $Y$.  This should be compared with \S0.3.2 of ~ibid., which gives a reconstruction result for Gromov--Witten invariants of $Y$ such that all insertions are pullbacks of cohomology classes on $X$.  One can use the same approach together with the Abelian/Non-Abelian Correspondence with bundles~\cite[\S6.1]{Ciocan-Fontanine--Kim--Sabbah} to determine the genus-zero Gromov--Witten invariants of many subvarieties of flag manifolds and partial flag bundles.  
\end{remark}

\begin{remark}
  The formulation in Theorem~\ref{thm:main} is well-suited to proving mirror theorems for toric complete intersections or subvarieties of flag manifolds.  One first  obtains a family $t \mapsto I_\be(t,z)$ of elements of $\cL_{\be}$, by combining the Mirror Theorem for toric varieties or toric Deligne--Mumford stacks~\cite{Givental:toric,CCIT:stacks_1,Cheong--Ciocan-Fontanine--Kim} with the Quantum Lefschetz theorem~\cite{Coates--Givental} or the Abelian/Non-Abelian Correspondence with bundles~\cite[\S6.1]{Ciocan-Fontanine--Kim--Sabbah}.  After taking the non-equivariant limit $\lambda \to 0$ and applying Theorem~\ref{thm:main}, one can then argue as in \cite[\S9]{Coates--Givental} or \cite[Example~9]{CCIT:stacks_2}.
\end{remark}

\section{The Proof of Theorem~\ref{thm:main}}

\subsection{The Non-Equivariant Limit Exists}

For the remainder of this note, we consider only stable maps of genus zero.  Since $E$ is convex, we have that $R^1 \pi_\star \ev^\star E = 0$ and hence that $E_{0,n+1,d}$ is a vector bundle.  The fiber of $E_{0,n+1,d}$ over a stable map $f \colon C \to X$ is $H^0(C,f^\star E)$, and thus there is an exact sequence of vector bundles:
\begin{equation}
  \label{def:E_prime}
  \xymatrix{
    0 \ar[rr] &&
    E_{0,n+1,d}' \ar[rr] &&
    E_{0,n+1,d} \ar[rr]^{\ev_{n+1}} &&
    \ev_{n+1}^\star E \ar[rr] &&
    0
  }
\end{equation}
This implies that $\be(E_{0,n+1,d}) = \be(E_{0,n+1,d}') \be(\ev_{n+1}^\star E)$.  The Projection Formula, together with the fact that $\phi^\epsilon = \phi^\epsilon_\be \be(E)$, gives that:
\[
\bJ_\be(\bt) = 
{-z} + \bt(z) + \sum
\frac{Q^d}{n! } 
(\ev_{n+1})_\star
\left[
  [X_{0,n+1,d}]^\vir \cap \be(E_{0,n+1,d}') \cup \psi_{n+1}^m \cup \prod_{i=1}^n \ev_i^\star t_{k_i} \cup \psi_i^{k_i}
\right] (-z)^{-m-1}
\]
This makes it clear that the non-equivariant limit $\bJ_\be(\bt)\big|_{\lambda = 0}$ exists.  Let us write $e(\cdot)$ for the non-equivariant Euler class, noting that $e(\cdot)$ is the non-equivariant limit of $\be(\cdot)$.

\subsection{A Comparison of Virtual Fundamental Classes}

Consider the diagram:
\begin{equation}
  \label{eq:main_diagram}
  \begin{aligned}
    \xymatrix{
      {\displaystyle \coprod_{\delta : i_\star \delta = d} Y_{0,n+1,\delta}} \ar[rr]^G \ar[dd]_{\ev} && Z \ar[rr]^F \ar[dd]_{\ev} && X_{0,n+1,d} \ar[dd]_{\ev} \\
      \\
      Y^{n+1} \ar[rrdd]_r \ar[rr]^g && X^n \times Y \ar[dd]_q \ar[rr]^f && X^{n+1} \ar[dd]_p \\
      \\
      && Y \ar[rr]^i && X
    }
  \end{aligned}
\end{equation}
where $p$,~$q$, and~$r$ are projections onto the last factor of their domains (which are products);  $f$~and~$g$ are induced by the inclusion $i\colon Y \to X$; the maps~$\ev$ in the first and third columns are the evaluation maps $\ev_1 \times \cdots \times \ev_{n+1}$; the upper right-hand square is Cartesian; the composition $G \circ F$ is the union of canonical inclusions $Y_{0,n+1,\delta} \to X_{0,n+1,d}$; and the map $G$ is defined by the universal property of the fiber product $Z$.  The stack $Z$ consists of those stable maps in $X_{0,n+1,d}$ such that the last marked point lies in $Y$; it is the zero locus of the section $\ev_{n+1}^\star s \in \Gamma(X_{0,n+1,d}, \ev_{n+1}^\star E)$.  The map $\ev$ in the second column is also given by $\ev_1 \times \cdots \times \ev_{n+1}$.

\begin{proposition}
  \label{pro:comparison}
  With notation as above, we have:
  \begin{enumerate}
  \item[(A)]
    \[    
    f^! \Big( e(E_{0,n+1,d}') \cap [X_{0,n+1,d}]^\vir \Big)
    =
    \sum_{\delta : i_\star \delta = d} G_\star [Y_{0,n+1,\delta}]^\vir
    \]
  \item[(B)] For any $(k_1,\ldots,k_{n+1}) \in \NN^{n+1}$:
    \[
    f^\star \ev_\star \Big( \psi_1^{k_1} \cup \cdots \cup \psi_{n+1}^{k_{n+1}} \cup e(E_{0,n+1,d}') \cap [X_{0,n+1,d}]^\vir \Big)
    =
    \sum_{\delta : i_\star \delta = d}
    g_\star \ev_\star 
    \Big( \psi_1^{k_1} \cup \cdots \cup \psi_{n+1}^{k_{n+1}} \cap [Y_{0,n+1,\delta}]^\vir \Big)
    \]
  \end{enumerate}
\end{proposition}

\begin{proof}
  Let $0_X \colon X_{0,n+1,d} \to E_{0,n+1,d}$, $0'_X \colon X_{0,n+1,d} \to E'_{0,n+1,d}$, $0'_Z \colon Z \to E'_{0,n+1,d}\big|_Z$ denote the zero sections.  \label{def:zero_sections} Consider the Cartesian diagram:
  \[
  \xymatrix{
    {\displaystyle \coprod_{\delta : i_\star \delta = d} Y_{0,n+1,\delta}}  \ar[rr]^G \ar[dd]_{G} && Z \ar[rr]^F \ar[dd]_{\tilde{s}|_Z} && X_{0,n+1,d} \ar[dddd]_{\tilde{s}} \\
    \\
    Z \ar[dd]_F \ar[rr]^{0_Z'} && E_{0,n+1,d}'\big|_Z \ar[dd] \\
    \\
    X_{0,n+1,d} \ar[rr]^{0_X'} && E_{0,n+1,d}' \ar[rr]^j &&  E_{0,n+1,d} 
  }
  \]
  where $j$ is the inclusion from \eqref{def:E_prime} and $\tilde{s}$ is the section of $E_{0,n+1,d}$ induced by the section $s \colon X \to E$ that defines $Y$.  Note that, on the bottom row, $0_X' \circ j = 0_X$.  We have:
  \begin{align*}
    \sum_{\delta : i_\star \delta = d} G_\star [Y_{0,n+1,\delta}]^\vir
    & = \sum_{\delta : i_\star \delta = d} G_\star 0_X^! [X_{0,n+1,d}]^\vir && \text{(functoriality~\cite{Kim--Kresch--Pantev})} \\
    & = \sum_{\delta : i_\star \delta = d} G_\star (0'_X)^! j^! [X_{0,n+1,d}]^\vir && \text{(functoriality~\cite[Theorem~6.5]{Fulton})} \\
    & = \sum_{\delta : i_\star \delta = d} (0'_X)^\star (\tilde{s}|_Z)_\star j^! [X_{0,n+1,d}]^\vir && \text{(by \cite[Theorem~6.2]{Fulton})} \\
    & = e\big(E_{0,n+1,d}'|_Z\big) \cap j^! [X_{0,n+1,d}]^\vir \\
    & = j^! \Big(e(E_{0,n+1,d}') \cap [X_{0,n+1,d}]^\vir \Big) \\
    & = f^! \Big(e(E_{0,n+1,d}') \cap [X_{0,n+1,d}]^\vir \Big) 
  \end{align*}
  This proves (A).  Since $f^\star \ev_\star = \ev_\star f^!$ \cite[Theorem~6.2]{Fulton} and $g_\star \ev_\star = \ev_\star G_\star$, and since the classes $\psi_i$ on $Z$ and on $Y_{0,n+1,\delta}$ are pulled back from the class $\psi_i$ on $X_{0,n+1,d}$, (A) implies (B).
\end{proof}

\subsection{Applying the Projection Formula}

We now deduce Theorem~\ref{thm:main} from Proposition~\ref{pro:comparison}.  This amounts to repeated application of the Projection Formula.  Recall the diagram \eqref{eq:main_diagram}.  The non-equivariant limit $\bJ_\be(\bt)\big|_{\lambda=0}$ is equal to: 
\begin{multline*}
  {-z} + \bt(z) + \sum
  \frac{Q^d}{n! } 
  (\ev_{n+1})_\star
  \left[
    [X_{0,n+1,d}]^\vir \cap e(E_{0,n+1,d}') \cup \psi_{n+1}^m \cup \prod_{i=1}^n \ev_i^\star t_{k_i} \cup \psi_i^{k_i}
  \right] (-z)^{-m-1} \\
  =
    {-z} + \bt(z) + \sum
  \frac{Q^d}{n! (-z)^{m+1}} 
  p_\star
  \left[
    \ev_\star \Big(
    [X_{0,n+1,d}]^\vir \cap e(E_{0,n+1,d}') \cup \psi_{n+1}^m \cup \prod_{i=1}^n \psi_i^{k_i}
    \Big)
    \cup \bigotimes_{i=1}^n t_{k_i}
  \right] 
\end{multline*}
Using $i^\star p_\star = q_\star f^\star$, we see that the pullback $i^\star \bJ_\be(\bt)\big|_{\lambda=0}$ is:
\[
 {-z} + i^\star \bt(z) + \sum
\frac{Q^d}{n! (-z)^{m+1}} 
q_\star
\left[
  f^\star \ev_\star \Big(
  [X_{0,n+1,d}]^\vir \cap e(E_{0,n+1,d}') \cup \psi_{n+1}^m \cup \prod_{i=1}^n \psi_i^{k_i}
  \Big)
  \cup \bigotimes_{i=1}^n t_{k_i}
\right] 
\]
Proposition~\ref{pro:comparison}(B) now gives:
\[
i^\star \bJ_\be(\bt)\big|_{\lambda=0} =
{-z} + i^\star \bt(z) + {\sum}'
\frac{Q^{i_\star \delta}}{n! (-z)^{m+1}} 
q_\star
\left[
  g_\star \ev_\star \Big(
  [Y_{0,n+1,\delta}]^\vir \cup \psi_{n+1}^m \cup \prod_{i=1}^n \psi_i^{k_i}
  \Big)
  \cup \bigotimes_{i=1}^n t_{k_i}
\right]
\]
where the sum $\sum'$ runs over non-negative integers~$n$ and~$m$, multi-indices $k = (k_1,\ldots,k_n)$ in $\NN^n$, degrees $\delta \in H_2(Y;\ZZ)$,  and basis indices $\epsilon$.  Applying the Projection Formula again, we see that:
\begin{align*}
  i^\star \bJ_\be(\bt)\big|_{\lambda=0} & =
  {-z} + i^\star \bt(z) + {\sum}'
  \frac{Q^{i_\star \delta}}{n! (-z)^{m+1}} 
  q_\star
  \left[
    g_\star \ev_\star \Big(
    [Y_{0,n+1,\delta}]^\vir \cup \psi_{n+1}^m \cup \prod_{i=1}^n \psi_i^{k_i}
    \Big)
    \cup \bigotimes_{i=1}^n t_{k_i}
  \right] \\
  & =
  {-z} + i^\star \bt(z) + {\sum}'
  \frac{Q^{i_\star \delta}}{n! (-z)^{m+1}} 
  q_\star g_\star
  \left[
    \ev_\star \Big(
    [Y_{0,n+1,\delta}]^\vir \cup \psi_{n+1}^m \cup \prod_{i=1}^n \psi_i^{k_i}
    \Big)
    \cup g^\star \bigotimes_{i=1}^n t_{k_i}
  \right] \\
  & =
  {-z} + i^\star \bt(z) + {\sum}'
  \frac{Q^{i_\star \delta}}{n! (-z)^{m+1}} 
  r_\star
  \left[
    \ev_\star \Big(
    [Y_{0,n+1,\delta}]^\vir \cup \psi_{n+1}^m \cup \prod_{i=1}^n \psi_i^{k_i}
    \Big)
    \cup \bigotimes_{i=1}^n i^\star t_{k_i}
  \right] \\
  & = \bJ_Y(i^\star \bt) \big|_{Q^\delta \mapsto Q^{i_\star \delta}}
\end{align*}
The Theorem is proved. \qed

\begin{remark}
  \label{rem:stacks}
  Let $X$ be a smooth Deligne--Mumford stack with projective coarse moduli space, let $E \to X$ be a convex vector bundle, let $Y$ be the substack in $X$ defined by a regular section of $E$, and let $i \colon IY \to IX$ be the map of inertia stacks induced by the inclusion $Y \to X$.  The analog of Theorem~\ref{thm:main} holds in this context, with the same proof: cf.~\cite[Proposition~2.4]{Iritani:periods}.  Note that a convex line bundle on a Deligne--Mumford stack is necessarily the pullback of a line bundle on the coarse moduli space \cite{CGIJJM}.
\end{remark}

\section*{Acknowledgements}

\noindent I thank Alessio Corti, Tom Graber, Hiroshi Iritani, Yuan-Pin Lee, and Andrea Petracci for useful conversations.

\bibliography{bibliography}{}
\bibliographystyle{plain}

\appendix

\section{Notation}
\label{sec:notation}

What follows is a list of notation and definitions: first for symbols in Roman font, then for Greek symbols, then for miscellaneous symbols.

\medskip

\begin{longtable}{lp{0.85\textwidth}}
  $\bc$ & an invertible multiplicative characteristic class \\
  $\be$ & the $S^1$-equivariant Euler class; see page~\pageref{def:S1_action} for the definition of the $S^1$-action\\
  $e$ & the non-equivariant Euler class \\
  $E$ & a convex vector bundle over $X$ \\
  $E_{g,n,d}$ & the twisting class $E_{g,n,d} \in K^0(X_{g,n,d})$; see page~\pageref{def:twisting_class} \\
  $E'_{0,n+1,d}$ & a sub-bundle of $E_{0,n+1,d}$; see page~\pageref{def:E_prime} \\
  $\ev_i$ & the evaluation map $X_{g,n,d} \to X$ at the $i$th marked point \\
  $\cH_X$, $\cH_Y$ & Givental's symplectic vector space; see page~\pageref{def:HX} \\
  $\cL_\be$ & Givental's Lagrangian cone for Euler-twisted invariants of $X$; see page~\pageref{def:LX_e} \\
  $\cL_X$, $\cL_Y$ & Givental's Lagrangian cone for $X$, $Y$; see page~\pageref{def:LX} \\
  $i$ & the inclusion map $Y \to X$ \\
  $j$ & the inclusion map $E_{0,n+1,d}' \to E_{0,n+1,d}$ \\
  $\bJ_\be(\bt)$ & a general point on $\cL_\be$; see \eqref{eq:Je} \\
  $\bJ_X(\bt)$ & a general point on $\cL_X$; see \eqref{eq:JX} \\
  $k_i$ & a non-negative integer \\
  $Q^d$ & the representative of $d \in H_2(X;\ZZ)$ in the Novikov ring $\Lambda_X$ \\
  $\bt$ & $\bt(z) = t_0 + t_1 z + t_2 z^2 + \cdots$ where $t_i \in H^\bullet(X)$ \\
  $t_i$ & a cohomology class on $X$ \\
  $X$ & a smooth projective variety \\
  $X_{g,n,d}$ & the moduli space of stable maps to $X$, from genus-$g$ curves with $n$~marked points, of degree $d \in H_2(X;\ZZ)$ \cite{Kontsevich:enumeration,Kontsevich--Manin}\\
  $[X_{g,n,d}]^\vir$ & the virtual fundamental class of the moduli space of stable maps to $X$ \cite{Li--Tian,Behrend--Fantechi} \\
  $Y$ & a subvariety of $X$ cut out by a regular section of $E$ \\
  $Y_{g,n,d}$ & the moduli space of stable maps to $Y$, from genus-$g$ curves with $n$~marked points, of degree $d \in H_2(Y;\ZZ)$ \cite{Kontsevich:enumeration,Kontsevich--Manin}\\
  $[Y_{g,n,d}]^\vir$ & the virtual fundamental class of the moduli space of stable maps to $Y$ \cite{Li--Tian,Behrend--Fantechi} \\ 
  \\
  $\gamma_i$ & a cohomology class on $X$ \\
  $\lambda$ & the generator of $H^\bullet_{S^1}\big(\{\text{point} \} \big)$ given by the first Chern class of $\cO(1) \to \CC P^\infty \cong BS^1$ \\
  $\Lambda_X$ & the Novikov ring of $X$; this is a completion of the group ring $\QQ\big[H_2(X;\ZZ)\big]$ with respect to the valuation $v(Q^d) = \int_d \omega$, where $Q^d$ is the representative of $d \in H_2(X;\ZZ)$ in the group ring and $\omega$ is the K\"ahler form on $X$ \\
  $\phi_\epsilon$ & an element of the basis $\{\phi_\epsilon\}$ for $H^\bullet(X;\QQ)$ \\
  $\phi^\epsilon$ & an element of the dual basis $\{\phi^\epsilon\}$ for $H^\bullet(X;\QQ)$, so that $(\phi_\mu,\phi^\nu) = \delta_\mu^\nu$ \\
  $\psi_i$ & the first Chern class of the universal cotangent line bundle $L_i \to X_{g,n,d}$ at the $i$th marked point \\
  $\Omega_X$, $\Omega_\be$, $\Omega_Y$ & the symplectic forms on $\cH_X$, $\cH_X$, and $\cH_Y$ respectively; see page~\pageref{def:Omega} \\
  \\
  $0_X$, $0_X'$, $0_Z'$ & zero section maps; see page~\pageref{def:zero_sections} \\
  $(\cdot,\cdot)$ & the Poincar\'e pairing on $H^\bullet(X)$, $(\alpha,\beta) = \int_X \alpha \cup \beta$ \\
  $(\cdot,\cdot)_\be$ & the twisted Poincar\'e pairing on $H^\bullet(X)$, $(\alpha,\beta) = \int_X \alpha \cup \beta \cup \be(E)$ \\
  $\correlator{\cdots \vphantom{\big|}}_{g,n,d}^X$ & Gromov--Witten invariants or twisted Gromov--Witten invariants of $X$; see (\ref{eq:GW_invariant}--\ref{eq:twisted_GW_invariant})
\end{longtable}

\end{document}